\documentclass[12pt]{amsart}
\usepackage{amsmath}
\usepackage{amscd}
\usepackage{amssymb}
\usepackage{amsfonts}
\usepackage{graphicx}
\usepackage{wrapfig}
\usepackage{amsthm}
\usepackage{euscript}
\usepackage[utf8]{inputenc}
\usepackage[english]{babel}
\usepackage{pgf,tikz}
\usepackage[colorlinks, citecolor=blue]{hyperref}

\newtheorem{theorem}{Theorem}[section]
\newtheorem{lemma}[theorem]{Lemma}
\newtheorem{corollary}[theorem]{Corollary}

\theoremstyle{definition}

\theoremstyle{remark}
\newtheorem{remark}[theorem]{Remark}
\numberwithin{equation}{section}

\usepackage[utf8]{inputenc}
\usepackage{tikz}
\usetikzlibrary{patterns}
\pgfdeclarepatternformonly{my dots}{\pgfqpoint{-1pt}{-1pt}}{\pgfqpoint{2.5pt}{2.5pt}}{\pgfqpoint{6pt}{6pt}}%
{
	\pgfpathcircle{\pgfqpoint{0pt}{0pt}}{.5pt}
	\pgfpathcircle{\pgfqpoint{7pt}{7pt}}{.5pt}
	\pgfusepath{fill}
}
\begin{document}

\title[Diameter estimates for submanifolds in manifolds]
{Diameter estimates for submanifolds in manifolds
 with nonnegative curvature}

\author{Jia-Yong Wu}
\address{Department of Mathematics, Shanghai University, Shanghai 200444, China;
and Newtouch Center for Mathematics of Shanghai University, Shanghai 200444, China}
\email{wujiayong@shu.edu.cn}

\subjclass[2010]{Primary 53C40; Secondary 53C42}
\dedicatory{}
\date{submitted to DGA: 30 Aug 2022; accepted: 10 Jul 2023}
\keywords{mean curvature; Riemannian submanifold; diameter estimate; surface with boundary}

\begin{abstract}
Given a closed connected manifold smoothly immersed in a complete noncompact
Riemannian manifold with nonnegative sectional curvature, we estimate the
intrinsic diameter of the submanifold in terms of its mean curvature field
integral. On the other hand, for a compact convex surface with boundary
smoothly immersed in a complete noncompact Riemannian manifold with
nonnegative sectional curvature, we can estimate its intrinsic diameter in
terms of its mean curvature field integral and the length of its boundary.
These results are supplements of previous work of Topping, Wu-Zheng and
Paeng.
\end{abstract}
\maketitle

\section{Introduction}\label{Int1}

A basic question in submanifold geometry is to understand the geometry
and topology of submanifolds. For example, Simon \cite{Si} proved an upper
estimate of the external diameter of a closed connected surface $\Sigma$ immersed
in the $3$-dimensional Euclidean space $\mathbb{R}^3$ in terms of its area
and Willmore energy. For an $m$-dimensional closed (compact, no boundary)
manifold $M$ smoothly immersed in $\mathbb{R}^n$, Topping \cite{To2}
applied Michael-Simon's inequality \cite{MS} to estimate its intrinsic
diameter in terms of its mean curvature field $H$ by
\begin{equation}\label{diaest}
d_{int}(M)\le C(m)\int_M |H|^{m-1}dv_M,
\end{equation}
where $dv_M$ is the Riemannian measure on $M$ induced by $\mathbb{R}^n$,
$d_{int}(M):=\sup_{x,y\in M}dist_M(x,y)$ and $dist_M(x,y)$ denotes
the distance in $M$ from $x$ to $y$.

Later, Wu-Zheng \cite{WZ} extended Topping's result to the case for
submanifolds of a general manifold by using Hoffman-Spruck's inequality
\cite{HS}. To state the result, we need to fix some notations. Let
$M\rightarrow N$ be an isometric immersion of Riemannian manifolds of
dimension $m$ and $n$, respectively. The mean curvature vector field of
the immersion is denoted by $H$. We denote the injectivity radius of
$N$ restricted to $M$ by $\bar{R}(M)$, and denote the volume of the
closed manifold $M$ by $V(M)$. Let $\omega_m$ denote the volume
of the unit ball in $\mathbb{R}^m$ and let $b$ be a positive real number
or a pure imaginary one. When the sectional curvature of $N$ satisfies
$K_N\le b^2$, Wu-Zheng \cite{WZ} showed that for any $0<\alpha<1$, if
\begin{equation}\label{condt}
b^2(1-\alpha)^{-2/m}\left(\omega_m^{-1}V(M)\right)^{2/m}\le 1
\quad\mathrm{and}\quad
2\rho_0\le\bar{R}(M),
\end{equation}
where
\begin{equation*}
\rho_0:=\left\{
\begin{aligned}
b^{-1}\sin^{-1}\left[b(1-\alpha)^{-1/m}\left(\omega_m^{-1}V(M)\right)^{1/m}\right]
\quad\quad&\mathrm{for}\,\,\, b\,\,\, \mathrm{real},\\
(1-\alpha)^{-1/m}\left(\omega_m^{-1}V(M)\right)^{1/m}
\quad\quad&\mathrm{for}\,\,\, b\,\,\, \mathrm{imaginary},
\end{aligned}
\right.
\end{equation*}
then there exists a constant $C(m,\alpha)$ depending on $m$ and
$\alpha$ such that
\[
d_{int}(M)\leq C(m,\alpha)\int_M |H|^{m-1}dv_M.
\]
In particular, when  $K_N\le 0$, then $\bar{R}(M)=+\infty$ and the
diameter estimate could hold without the complicated condition \eqref{condt}.

In \cite{Pa}, Paeng generalized the above diameter estimate to an immersed
surface with boundary. Precisely, for a compact $2$-dimensional convex surface
$\Sigma$ with boundary $\partial\Sigma$ immersed in an $n$-dimensional
simply connected manifold $N$ with $K_N\le 0$, he proved that
\begin{equation}\label{diaest2}
d_{int}(\Sigma)\le 432\pi\left(\int_{\Sigma} |H|dv_{\Sigma}
+\frac{L(\partial\Sigma)}{2}\right),
\end{equation}
where $L(\partial\Sigma)$ is the length of curve $\partial\Sigma$.
However there seem to be essential obstacles to generalize his
result to a general submanifold.

In this paper, we will generalize the above results to the noncompact
manifold with nonnegative sectional curvature. Let $(N,g)$ be an
$n$-dimensional complete noncompact manifold with nonnegative
Ricci curvature. For any base point $p\in N$, let $B(p,r)\subset N$ denote a
geodesic ball with center $p$ and intrinsic radius $r>0$, and its volume is given by
\[
V(p,r):=\int_{B(p,r)}dv_N.
\]
The \emph{asymptotic volume ratio} of $(N,g)$ is defined by
\[
\theta:=\lim_{r\to\infty}\frac{V(p,r)}{\omega_nr^n}.
\]
Here the limit $\theta$ always exists and it does not depend on the choice
of point $p$. By the Bishop-Gromov volume comparison, we have $\theta\le 1$.
When $N$ is the Euclidean space, $\theta=1$.

In the first part of this paper, we prove that
\begin{theorem}\label{Main}
For $m\ge 1$, let $M$ be an $m$-dimensional closed connected manifold smoothly
isometrically immersed in an $n$-dimensional simply connected, complete noncompact,
Riemannian manifold $(N,g)$ with nonnegative sectional curvature ($K_N\ge 0$).
There exists a constant $C(n,m,\theta)$ depending only on $n$, $m$ and $\theta$
such that
\[
d_{int}(M)\le C(n,m,\theta)\int_M |H|^{m-1}dv_M,
\]
where $\theta$ denotes the asymptotic volume ratio of $(N,g)$.
\end{theorem}

\begin{remark}
In Theorem \ref{Main}, we require $\theta\neq0$ 
otherwise our diameter estimate is not meaningful. In our proof, one can take
\begin{equation*}
C(n,m,\theta)=4\delta^{1-m},\quad \mathrm{where}\,\,\delta=\left\{
\begin{aligned}
\min\{\tfrac{n\omega_n\theta}{2^m(n-m)\omega_{n-m}},\,\, \omega_m,\,\, 1\}, \qquad & n\ge m+2,\\
\min\{\tfrac{(m+2)\omega_{m+2}\theta}{2^{m+1}\pi},\,\, \omega_m,\,\, 1\}, \qquad & n=m+1.
\end{aligned}
\right.
\end{equation*}
Here we only provide an explicit constant of
$C(n,m,\theta)$ and omit its optimal discussion.
\end{remark}

\begin{remark}
Theorem \ref{Main} is not completely covered by the result of Wu-Zheng
\cite{WZ}. Indeed, letting $r(x):=r(x,p)$ denote the distance function from
a fixed point $p$ to $x\in N$, if the noncompact ambient manifold $N$
satisfies $0<K_N\le\frac{c}{r(x)}$ for some constant $c>0$, where
$r(x)\ge 1$, then this case is adapted to Theorem \ref{Main}, but it may be
not suitable for the condition \eqref{condt}. So Theorem \ref{Main} can be
regarded as a supplement of Wu-Zheng's result.
\end{remark}

In the second part of this paper, we prove that

\begin{theorem}\label{Main2}
Let $\Sigma$ be a compact convex $2$-dimensional surface with boundary
$\partial\Sigma$ smoothly isometrically immersed in an $n$-dimensional simply
connected, complete noncompact, Riemannian manifold $(N,g)$ with nonnegative
sectional curvature ($K_N\ge 0$). Then there exists a constant $C(n,\theta)$
depending only on $n$ and $\theta$ such that
\begin{equation}\label{diasur}
d_{int}(\Sigma)\le C(n,\theta)\left(\int_{\Sigma} |H|dv_{\Sigma}
+L(\partial\Sigma)\right),
\end{equation}
where $d_{int}(\Sigma):=\sup_{x,y\in \Sigma}dist_\Sigma(x,y)$ and
$L(\partial\Sigma)$ is the length of curve $\partial\Sigma$.
In particular, we can take $C(n,\theta)=\frac{9(n-2)}{2\pi\theta}$
when $n\ge 4$, and $C(3,\theta)=\frac{9}{\pi\theta}$.
\end{theorem}

\begin{remark}\label{cofsma}
Paeng \cite{Pa} proved a similar intrinsic diameter estimate for convex
surfaces with boundary immersed in a manifold with nonpositive sectional
curvature. Theorem \ref{Main2} can be viewed as the supplement of Paeng's
result for the positive curvature case.
\end{remark}

Recently, Miura \cite{Mi} gave a geometric argument for bounding the diameter
of a connected compact surface with boundary of arbitrary codimension in
Euclidean space. This was further generalized by Flaim-Scharrer \cite{FS}
to a conformally flat ambient space. For a general surface with boundary,
Menne-Scharrer \cite{MeS} provided a diameter bound in a fairly general
framework of varifolds.

The proof strategy of Theorems \ref{Main} and \ref{Main2} stems from
\cite{To2}, \cite{WZ} and \cite{Pa}, but we need to use Brendle's
Sobolev inequality \cite{Br} instead of Hoffman-Spruck's Sobolev
inequality \cite{HS}. The cores of our proofs are two alternative
lemmas which say that the maximal function (with boundary) and
volume ratio cannot be simultaneously bounded below by a fixed constant
(see Lemma \ref{lemm1} and Lemma \ref{lemb}). Two lemmas are proved
by choosing a suitable cut-off function in the recent Brendle's Sobolev
inequality \cite{Br}. Once these lemmas are acquired, the theorems easily
follow by the classical Vitali-type covering argument.

Similar to the above proof strategy, there have been other diameter upper
estimates for Riemannian manifolds in terms of certain scalar curvature
integral under different settings; see Fu-Wu \cite{FW}, Topping \cite{To1},
Wu \cite{Wu}, Zhang \cite{Zhq} for details.

The rest of this paper is organized as follows. In Section \ref{sec2}, we
will introduce the maximal function and the volume ratio and give an
alternative lower bound between them (see Lemma \ref{lemm1}). In Section
\ref{sec3}, we will apply the alternative lemma to prove Theorem \ref{Main}.
As a application, a compact theorem relating to diameter estimate will be
provided. In Section \ref{sec4}, we will prove Theorem \ref{Main2} by
following the argument of Theorem \ref{Main}.


\section{Maximal function and volume ratio}\label{sec2}

In this section, we first recall Dong-Lin-Lu's Sobolev inequality for submanifolds
of a Riemannian manifold with asymptotically nonnegative curvature in \cite{DLL},
which generalizes Brendle's result in \cite{Br}.
Then we introduce the maximal function and the volume ratio for submanifolds
of a Riemannian manifold. Finally we apply Dong-Lin-Lu's Sobolev inequality to study a
lower bound between the maximal function and the volume ratio.

Recall that Dong-Lin-Lu \cite{DLL} applied Brendle's argument \cite{Br} to get
the following Sobolev inequality for submanifolds in a Riemannian manifold
with asymptotically nonnegative sectional curvature.
\begin{theorem}\label{SI}
Let $N$ be an $n$-dimensional complete noncompact Riemannian manifold $(N,g)$
of asymptotically nonnegative sectional curvature with respect to a base point
$o\in N$. Let $M$ be an $m$-dimensional
compact submanifold (possibly with boundary $\partial M$), and let $f$ be a positive
smooth function on $M$. If $n\ge m+2$, then
\begin{equation*}
\begin{aligned}
\left(\int_M f^{\frac{m}{m-1}}dv_M\right)^{\frac{m-1}{m}}&\le c(n,m)\theta^{-\frac 1m}
\left(\frac{e^{2r_0b_1+b_0}}{1+b_0}\right)^{\frac{n-1}{m}}\\
&\quad\times\left(\int_M\sqrt{|\nabla f|^2+f^2|H|^2}dv_M+\int_{\partial M}fdv_{\partial M}+2nb_1\int_Mfdv_M\right),
\end{aligned}
\end{equation*}
where $r_0=\max\{d(o,x)|x\in M\}$,
$\theta$ denotes the asymptotic volume ratio of $(N,g)$, $dv_M$ and
$dv_{\partial M}$ are the measures induced by the ambient space $(N,g)$.
Here constant $c(n,m)$ is given by
\begin{equation}\label{condef}
c(n,m)=\frac{1}{m}\left(\frac{(n-m)\omega_{n-m}}{n\omega_n}\right)^{\frac 1m}.
\end{equation}
\end{theorem}

\begin{remark}
For $n=m+2$, Brendle \cite{Br} confirmed that the above Sobolev inequality is sharp.
For $n=m+1$, Theorem \ref{SI} remains true but with a larger constant $c(n,m)$.
Indeed, if $M$ is an $m$-dimensional submanifold of an $(m+1)$-dimensional
manifold $N$, then we can regard $M$ as a submanifold of the $(m+2)$-dimensional
manifold $N\times \mathbb{R}$. Moreover, the product $N\times \mathbb{R}$ has
the same asymptotic volume ratio as $N$. Hence when $n=m+1$, the constant
$c(n,m)$ can be taken as
\begin{equation}\label{condef2}
c(n,m)=
\frac{1}{m}\left(\frac{2\pi}{(m+2)\omega_{m+2}}\right)^{\frac 1m}.
\end{equation}
\end{remark}

Below we shall follow the argument of \cite{To1,To2} (see also \cite{WZ}) and
prove an alternative lemma which states that the maximal function and the volume
ratio cannot simultaneously have a lower bound. Recall that for an
$m$-dimensional ($m\ge 2$) manifold $M$ smoothly isometrically immersed in an
$n$-dimensional Riemannian manifold $(N,g)$, for any $x\in M$ and $R>0$, the
\emph{maximal function} is defined by
\begin{equation}\label{defin1}
M(x,R):=\sup_{r\in(0,R]}r^{-\frac{1}{m-1}}\big[V(x,r)\big]
^{-\frac{m-2}{m-1}}\int_{B(x,r)}(|H|+2nb_1)dv_M,
\end{equation}
and the \emph{volume ratio} is defined by
\begin{equation}\label{defin2}
\kappa(x,R):=\inf_{r\in(0,R]}\frac{V(x,r)}{r^m}.
\end{equation}

We now apply Theorem \ref{SI} to show that the maximal function and the
volume ratio cannot be simultaneously smaller than a fixed constant. It
generalizes previous results of Topping \cite{To2} and Wu-Zheng \cite{WZ}.

\begin{lemma}\label{lemm1}
Let $M$ be an $m$-dimensional ($m\ge2$) complete manifold smoothly
isometrically immersed in an $n$-dimensional complete noncompact Riemannian
manifold $(N,g)$ with nonnegative sectional curvature. For any $n\ge m+1$,
there exists a constant $\delta>0$ depending only on $m$ and $n$, such that
for any $x\in M$ and $R>0$, at least one of the following is true:
\begin{enumerate}
  \item $M(x,R)\ge\delta$;
\item $\kappa(x,R)\ge\delta$.
\end{enumerate}
In particular, we can take $\delta=\min\{\frac{\theta}{(2m)^m c(n,m)^m},\,\,\omega_m,\,\, 1\}$.
When $m=2$, constant $\delta$ can be improved that
$\delta=\min\{\frac{\theta}{9 c(n,2)^2},\,\, \pi\}$. Here $c(n,m)$ is given
by \eqref{condef} if $n\ge m+2$ and \eqref{condef2} if $n=m+1$.
\end{lemma}

\begin{proof}[Proof of Lemma~\ref{lemm1}]
The proof is nearly the same as Topping's argument \cite{To2}. We
give a detailed proof for providing an explicit expression of $\delta$.

Suppose that $M(x,R)<\delta$ for some constant $\delta>0$, which
will be determined later. Then we want to prove $\kappa(x,R)>\delta$.
Since $M(x,R)<\delta$, its definition implies that
\begin{equation}\label{equ1}
\int_{B(x,r)}(|H|+2nb_1)dv_M<\delta
r^{\frac{1}{m-1}}\big[V(x,r)\big]^{\frac{m-2}{m-1}}
\end{equation}
for all $r\in(0,R]$. For the fixed point $x$, $V(r):=V(x,r)$ is a locally
Lipschitz function of $r>0$. Indeed, by the volume comparison theorem, one
could locally give an upper bound for the Lipschitz constant in terms of a
local lower bound for the Ricci curvature. In particular, $V(r)$ is
differentiable for almost all $r>0$. For such $r\in(0,R]$, and any $s>0$,
we introduce a Lipschitz cut-off function $h$ on $M$ satisfying
\begin{equation}\label{constr}
h(y)=\left\{
\begin{aligned}
1, \qquad & y\in B(x,r),\\
1-\frac{1}{s}\left(dist_M(x,y)-r\right), \qquad & y\in B(x,r+s)\setminus B(x,r),\\
0, \qquad & y\not\in B(x,r+s).
\end{aligned}
\right.
\end{equation}
Then we substitute $h$ of \eqref{constr} into Theorem \ref{SI}
and get that
\begin{equation*}
\begin{aligned}
V(r)^{(m-1)/m}&\le\left(\int_{M}h^{m/(m-1)}dv_M\right)^{(m-1)/m}\\
&\le c(n,m)\theta^{-\frac 1m}\left(\frac{e^{2r_0b_1+b_0}}{1+b_0}\right)^{\frac{n-1}{m}}
\left(\frac{V(r+s)-V(r)}{s}
+\int_{B(x,r+s)}|H|dv_M+2nb_1V(r+s)\right),
\end{aligned}
\end{equation*}
where the constant $c(n,m)$ is given by \eqref{condef} if $n\ge m+2$ and
\eqref{condef2} if $n=m+1$. Letting $s\downarrow 0$ in the above inequality
and combining \eqref{equ1}, we have
\begin{equation}\label{budesh1}
\frac{dV}{dr}+\delta r^{\frac{1}{m-1}}V(r)^{\frac{m-2}{m-1}}
-c(n,m)^{-1}\theta^{\frac1m}\left(\frac{1+b_0}{e^{2r_0b_1+b_0}}\right)^{\frac{n-1}{m}}
V(r)^{\frac{m-1}{m}}>0.
\end{equation}

On the other hand, we define a smooth function
\[
v(r):=\delta r^m,
\]
where $\delta>0$ is a small constant so that $\delta<\omega_m$.
Direct computation gives
\[
\frac{dv}{dr}+\delta r^{\frac{1}{m-1}}v(r)^{\frac{m-2}{m-1}}-
c(n,m)^{-1}\theta^{\frac 1m}v(r)^{\frac{m-1}{m}}
=\left(m\delta+\delta^{\frac{2m-3}{m-1}}
-c(n,m)^{-1}\theta^{\frac 1m}\delta^{\frac{m-1}{m}}\right)r^{m-1}.
\]
So, if $\delta$ is sufficiently small, depending
only on $m$, $n$ and $\theta$, such that
\[
m\delta+\delta^{\frac{2m-3}{m-1}}-c(n,m)^{-1}\theta^{\frac 1m}\delta^{\frac{m-1}{m}}\le0,
\]
which is equivalent to
\begin{equation}\label{condelta}
m\delta^{\frac 1m}+\delta^{\frac1m+\frac{m-2}{m-1}}\le c(n,m)^{-1}\theta^{\frac 1m},
\end{equation}
then we conclude
\begin{equation}\label{budesh2}
\frac{dv}{dr}+\delta r^{\frac{1}{m-1}}v(r)^{\frac{m-2}{m-1}}-
c(n,m)^{-1}\theta^{\frac 1m}v(r)^{\frac{m-1}{m}}\le 0.
\end{equation}

Indeed we can give an explicit expression of $\delta$ such that
\eqref{condelta} holds. Observe that if $\delta<1$, then
$\delta^{\frac1m+\frac{m-2}{m-1}}\le m\delta^{\frac 1m}$,
and hence if $2m\delta^{\frac 1m}\le c(n,m)^{-1}\theta^{\frac 1m}$,
then \eqref{condelta} and \eqref{budesh2} both hold. From this
observation, constant $\delta$ can be chosen as
\[
\delta=\min\left\{\left(\tfrac{1}{2m\, c(n,m)}\right)^m \theta,\,\, \omega_m,\,\, 1\right\}.
\]
In particular, when $m=2$, then $\omega_2=\pi$ and directly analyzing \eqref{condelta}
we can take
\[
\delta=\min\left\{\tfrac{\theta}{9\, c(n,2)^2},\,\, \pi \right\}.
\]

Notice that $\frac{V(r)}{r^m}\rightarrow \omega_m$ as $r\downarrow0$, while
$\frac{v(r)}{r^m}=\delta<\omega_m$, which gives $V(r)>v(r)$ as $r\downarrow0$.
Combining this with \eqref{budesh1} and~\eqref{budesh2}, we claim that
$V(r)>v(r)$ for all $r\in (0,R]$. Otherwise, there exists $r_0\in (0,R]$
such that $V(r_0)=v(r_0)$ and $V(r)>v(r)$ for all $r\in(0,r_0)$. Then
from \eqref{budesh1} and~\eqref{budesh2}, we can derive
$\tfrac{dV}{dr}|_{r=r_0}>\tfrac{dv}{dr}|_{r=r_0}$ and hence
\[
\frac{dV}{dr}>\frac{dv}{dr}
\]
in any sufficiently small neighborhood of $r_0$, which is impossible since
$V(r_0)=v(r_0)$ and $V(r)>v(r)$ for all $r\in(0,r_0)$. Hence the
claim follows. Thus,
\[
\kappa(x,R):=\inf_{r\in(0,R]}\frac{V(x,r)}{r^m}\ge\inf_{r\in(0,R]}\frac{v(r)}{r^m}=\delta,
\]
and the lemma follows.
\end{proof}


\section{Proof of Theorem \ref{Main}}\label{sec3}

In this section, we start to recall a classical Vitali-type covering property.
\begin{lemma}\label{cover}
Let $\gamma$ be a shortest geodesic joining any two points $x$ and $y$ in
a complete Riemannian manifold $(M,g)$, and $s$ denote a non-negative
bounded function defined on $\gamma$. If
$\gamma\subset\left\{B(p,s(p))~|~p\in \gamma\right\}$, then for any  $\rho\in(0,1/2)$,
there exists a countable (possibly finite) set of points $\{p_i\in\gamma\}$ such that
\begin{enumerate}
 \item $B(p_i,s(p_i))$ are disjoint;

 \item $\gamma\subset\cup_i B(p_i,s(p_i))$;

 \item $\rho\,dist_M(x,y)\le \sum_i 2s(p_i)$, where $dist_M(x,y)$ denotes
 the distance in $(M,g)$ from $x$ to $y$.
\end{enumerate}
\end{lemma}

Then we apply Lemma \ref{cover} and Lemma \ref{lemm1} to prove Theorem \ref{Main}.
Similar discussion is also used in \cite{FW}, \cite{To1}, \cite{To2}, \cite{Wu},
\cite{WZ} and \cite{Zhq}, etc.

\begin{proof}[Proof of Theorem \ref{Main}]
Since $m=1$ is trivial, we assume $m\ge 2$. Let $\delta$ be the fixed
constant determined in Lemma~\ref{lemm1}. Since $M$ is closed, then
$V(M)<+\infty$. So we may choose $R>0$ sufficiently large so that
$V(M)<\delta R^m$ because $\delta$ does not depend on $R$.
In particular, for any point $p\in M$,
\[
\kappa(p,R)\le\frac{V(p,R)}{R^m}\le\frac{V(M)}{R^m}<\delta.
\]
Then by Lemma~\ref{lemm1}, we have $M(p,R)\ge\delta$. By the definition
of $M(p,R)$, we know that there exists $s=s(p)>0$, where $s\le R$, such that
\begin{equation}\label{defMR}
\begin{aligned}
\delta&\le s^{-\frac{1}{m-1}}\left[V(p,s)\right]^{-\frac{m-2}{m-1}}
\int_{B(p,s)}|H|dv_M\\
&\le s^{-\frac{1}{m-1}}\left(\int_{B(p,s)}|H|^{m-1}dv_M\right)^{\frac{1}{m-1}},
\end{aligned}
\end{equation}
where in the second inequality we have used the H\"older inequality.
Therefore,
\begin{equation}\label{intine}
s(p)<\delta^{1-m}\int_{B(p,s(p))}|H|^{m-1}dv_M.
\end{equation}

Below we shall pick appropriate points $p$ at which to apply \eqref{intine} repeatedly.
Since $M$ is closed, we may assume $p_1$ and $p_2$ are two extremal points in
$M$ such that $d_{int}(M)=dist_M(p_1,p_2)$. Let $\gamma$ be a shortest geodesic
connecting $p_1$ and $p_2$. Clearly, the geodesic $\gamma$ is covered by the balls $\left\{B(p,s(p))~|~p\in\gamma\right\}$. By Lemma \ref{cover}, there exists
a countable (possibly finite) set of points $\left\{p_i\in\gamma\right\}$ such
that the geodesic balls $\{B(p_i,s(p_i))\}$ are disjoint and
\[
\rho\,d_{int}(M)=\rho\,dist_M(p_1,p_2)\le \sum_i 2s(p_i).
\]
Combining this with \eqref{intine},
\begin{equation}
\begin{aligned}\label{precisein}
d_{int}(M)&\leq\frac{2}{\rho}\sum_is(p_i)\\
&<\frac{2}{\rho}\delta^{1-m}\sum_i\int_{B(p_i,s(p_i))}|H|^{m-1}dv_M\\
&\le \frac{2}{\rho}\delta^{1-m}\int_M|H|^{m-1}dv_M,
\end{aligned}
\end{equation}
where $\delta>0$ is a fixed constant chosen as in Lemma~\ref{lemm1}.
Letting $\rho\to\frac{1}{2}-$ yields that
\begin{equation}\label{dediam}
d_{int}(M)\le4\delta^{1-m}\int_M|H|^{m-1}dv_M,
\end{equation}
which proves the theorem.
\end{proof}

As a application, the proof of Theorem \ref{Main} implies the following compact theorem.
\begin{corollary}\label{comp}
For $m\geq 1$, let $M$ be an $m$-dimensional complete manifold smoothly
isometrically immersed in an $n$-dimensional simply connected, complete,
nonnegative curved Riemannian manifold $(N,g)$ ($\mathrm{Sect}_N\ge 0$).
If there exists a point $p\in M$ such that
\[
V(p,R)=o(R^m)
\]
and
\[
\Lambda:=\int_M|H|^{m-1}dv_M<+\infty,
\]
then $(M,g)$ is compact.
\end{corollary}

\begin{proof}[Proof of Corollary \ref{comp}]
Since $V(p,R)=o(R^m)$, then for any sufficient large $R>0$ such that
\[
\frac{V(p,R)}{R^m}<\delta
\]
for some small constant $\delta$ which does not depend on $R$.
For example, $\delta$ can be chosen as in Lemma \ref{lemm1}.
Then following the same argument of Theorem \ref{Main}, we obtain that
\begin{equation*}
\begin{aligned}
d_{int}(B(p,R))&\le4\delta^{1-m}\int_{B(p,2R)}|H|^{m-1}dv_M\\
&\le 4\delta^{1-m}\int_M|H|^{m-1}dv_M,
\end{aligned}
\end{equation*}
where $d_{int}(B(p,R)):=\sup_{x,y\in B(p,R)}dist_M(x,y)$.
Since $\Lambda<+\infty$ by assumption, we indeed show that, for any
sufficient large $R>0$,
\[
d_{int}(B(p,R))\le C(m,\delta,\Lambda)
\]
for some constant $C(m,\delta,\Lambda)$ depending only on $m$, $\delta$ and
$\Lambda$. Thus, the conclusion follows by letting $R\to+\infty$.
\end{proof}

\section{Proof of Theorem \ref{Main2}}\label{sec4}
In this section, we mainly follow the argument of \cite{Pa} to prove
Theorem \ref{Main2}.

Let $\Sigma$ be a compact convex surface with the boundary $\partial\Sigma$
immersed in an $n$-dimensional simply connected manifold $(N,g)$ with $K_N\ge 0$.
For any interior point $x\in \Sigma\setminus\partial\Sigma$ and $R>0$, the
\emph{maximal function with boundary} is defined by
\begin{equation}\label{defina}
M(x,R):=\sup_{r\in(0,R]}r^{-1}\left(\int_{B(x,r)\cap\Sigma}|H|dv_\Sigma+L\left(B(x,r)\cap \partial\Sigma\right)\right)
\end{equation}
and the volume ratio is defined by
\begin{equation}\label{definb}
\kappa(x,R):=\inf_{r\in(0,R]}\frac{A(B(x,r)\cap\Sigma)}{r^2}.
\end{equation}
where $A(B(x,r)\cap\Sigma)$ denotes the volume of the set $B(x,r)\cap\Sigma$.

We will apply Theorem \ref{SI} to show that the maximal function with
boundary and the volume ratio cannot be simultaneously smaller than a fixed
constant. It extends Paeng's result \cite{Pa} to the case of nonnegative
sectional curvature.

\begin{lemma}\label{lemb}
Under the same assumptions as above, for any $n\ge 3$,
there exists a constant $\delta>0$ depending only on $n$, such that
for any interior point $x\in \Sigma$ and $R>0$, at least one of the following is true:
\begin{enumerate}
  \item $M(x,R)\ge\delta$;
\item $\kappa(x,R)\ge\delta$.
\end{enumerate}
In particular, we can take $\delta=\frac{8\pi\theta}{9(n-2)}$ if $n\ge 4$,
and $\delta=\frac{4\pi\theta}{9}$ if $n=3$.
\end{lemma}

\begin{proof}[Proof of Lemma~\ref{lemb}]
For any interior point $x\in\Sigma$ and $R>0$, suppose that $M(x,R)<\delta$
for some fixed constant $\delta>0$. That is, for all $r\in(0,R]$,
\begin{equation}\label{equ1b}
\int_{B(x,r)\cap\Sigma}|H|dv_\Sigma+L\left(B(x,r)\cap\partial\Sigma\right)<\delta r.
\end{equation}
For fixed interior point $x\in\Sigma$, $A(r):=A(B(x,r)\cap\Sigma)$ is a
locally Lipschitz function of $r>0$. Indeed, since $\Sigma$ is convex,
for any $\Delta r\in \mathbb{R}$, we see that
\begin{equation*}
\begin{aligned}
|A(r+\Delta r)- A(r)|:=&|A(B(x,r+\Delta r)\cap\Sigma)-A(B(x,r)\cap\Sigma)|\\
=&\left|A\left(\big(B(x,r+\Delta r)\backslash B(x,r)\big)\cap\Sigma\right)\right|\\
\le& |A(B(x,r+\Delta r)- A(B(x,r))|.
\end{aligned}
\end{equation*}
Also, we know $A(B(x,r))$ is locally Lipschitz of $r>0$ in the proof of
Lemma~\ref{lemm1}. Combining two facts, we immediately conclude that
$A(r)$ is a locally Lipschitz function of $r>0$. In particular, $A(r)$ is
differentiable for almost all $r>0$. For such $r\in(0,R]$ and any $s>0$, we
introduce a cut-off function $h$ on $\Sigma$ satisfying
\begin{equation}\label{constrb}
h(y)=\left\{
\begin{aligned}
1, \qquad & y\in B(x,r)\cap\Sigma,\\
1-\frac{1}{s}\left(d_\Sigma(x,y)-r\right),\qquad &y\in (B(x,r+s)\cap\Sigma)\setminus (B(x,r)\cap\Sigma),\\
0, \qquad & y\not\in B(x,r+s)\cap\Sigma.
\end{aligned}
\right.
\end{equation}
Since $\Sigma$ is convex, distance function $d_\Sigma(x,)$ is differentiable
almost everywhere. So $h$ is differentiable almost everywhere on
$\Sigma$. Substituting $h$ into Theorem \ref{SI}, we get
\begin{equation*}
\begin{aligned}
A(r)^{1/2}&\le\left(\int_{\Sigma}h^2dv_\Sigma\right)^{1/2}\\
&\le c(n,2)\theta^{-\frac 12}\left(\frac{A(r+s)-A(r)}{s}
+\int_{B(x,r+s)\cap\Sigma}|H|dv_\Sigma+L\left(B(x,r)\cap \partial\Sigma\right)\right),
\end{aligned}
\end{equation*}
where $c(n,2)$ is defined by \eqref{condef} if $n\ge 4$ and \eqref{condef2}
if $n=3$. Letting $s\downarrow 0$ and combining \eqref{equ1b} gives
\begin{equation}\label{budesh1b}
\frac{dA}{dr}+\delta r-
c(n,2)^{-1}\theta^{\frac 12}A(r)^{\frac{1}{2}}>0.
\end{equation}

On the other hand, we define another smooth function $v(r):=\delta r^2$,
where $0<\delta<\pi$ is a small constant. Direct computation gives
\[
\frac{dv}{dr}+\delta r-c(n,2)^{-1}\theta^{\frac 12}v(r)^{\frac{1}{2}}
=\left(3\delta-c(n,2)^{-1}\theta^{\frac 12}\delta^{\frac{1}{2}}\right)r.
\]
So, if $\delta=\min\{\tfrac{\theta}{9\, c(n,2)^2},\,\, \pi\}$,
then
\begin{equation}\label{budesh2b}
\frac{dv}{dr}+\delta r-c(n,2)^{-1}\theta^{\frac 12}v(r)^{\frac{1}{2}}\le 0.
\end{equation}

Since $\frac{A(r)}{r^2}\rightarrow\pi$ as $r\downarrow0$, while
$\frac{v(r)}{r^2}=\delta<\pi$, this implies $A(r)>v(r)$ as $r\downarrow0$.
Combining this with \eqref{budesh1b} and~\eqref{budesh2b}, we claim that
$A(r)>v(r)$ for all $r\in (0,R]$. Otherwise, there exists $r_0\in (0,R]$
such that $A(r_0)=v(r_0)$ and $A(r)>v(r)$ for all $r\in(0,r_0)$. Then
from~\eqref{budesh1b} and~\eqref{budesh2b}, we can derive
$\tfrac{dA}{dr}|_{r=r_0}>\tfrac{dv}{dr}|_{r=r_0}$ and hence
\[
\frac{dA}{dr}>\frac{dv}{dr}
\]
in any sufficiently small neighborhood of $r_0$, which is impossible since
$A(r_0)=v(r_0)$ and $A(r)>v(r)$ for all $r\in(0,r_0)$. Hence the
claim follows. Thus,
\[
\kappa(x,R):=\inf_{r\in(0,R]}\frac{A(B(x,r)\cap\Sigma)}{r^2}
\ge\inf_{r\in(0,R]}\frac{v(r)}{r^2}=\delta,
\]
and the lemma follows.
\end{proof}

With the help of Lemma~\ref{lemb}, we are able to prove Theorem \ref{Main2}
using a similar argument of Theorem \ref{Main}.
\begin{proof}[Proof of Theorem \ref{Main2}]
Let $\delta$ be the fixed constant given in Lemma~\ref{lemb}. Since $\Sigma$
is compact, then $A(\Sigma)<+\infty$. So, we can choose $R>0$ sufficiently
large so that $A(\Sigma)<\delta R^2$. In particular, for any point
$p\in\Sigma\setminus\partial\Sigma$,
\[
\kappa(p,R)\le \frac{A(B(p,R)\cap\Sigma)}{R^2}\le\frac{A(\Sigma)}{R^2}<\delta.
\]
Then by Lemma~\ref{lemb}, we get $M(p,R)\ge\delta$. This shows that
there exists $s=s(p)>0$, where $s\le R$, such that
\[
\delta\le s^{-1}
\left(\int_{B(p,s)\cap\Sigma}|H|dv_\Sigma+L(B(p,s)\cap \partial\Sigma)\right).
\]
Namely,
\begin{equation}\label{intineb}
s(p)<\delta^{-1}\left(\int_{B(p,s(p))\cap\Sigma}|H|dv_\Sigma+L(B(p,s(p))\cap \partial\Sigma)\right).
\end{equation}

On the other hand, since $\Sigma$ is compact, by an approximate argument,
we may assume $p_1$ and $p_2$ are two interior points of $\Sigma$ such that $d_{int}(\Sigma)=dist_\Sigma(p_1,p_2)$. Indeed, if $p_1, p_2\in \partial\Sigma$,
there exist interior points ${p_1}_i$ and ${p_2}_i$ such that ${p_1}_i\to p_1$ and
${p_2}_i\to p_2$. So $dist_\Sigma(p_1,p_2)=\lim_{i\to\infty}({p_1}_i,{p_2}_i)$.
Let $\gamma$ be a shortest geodesic connecting $p_1$ and $p_2$. Since $\Sigma$ is
convex, then $\gamma\subset \Sigma\setminus\partial\Sigma$ and it is covered by balls $\left\{B(p,s(p))~|~p\in\gamma\right\}$ of $\Sigma\setminus\partial\Sigma$.
By Lemma \ref{cover}, there exists
a countable (possibly finite) set of points $\left\{p_i\in\gamma\right\}$ such
that geodesic balls $\{B(p_i,s(p_i))\}$ are disjoint and
\[
\rho\,d_{int}(\Sigma)=\rho\,dist_\Sigma(p_1,p_2)\le \sum_i 2s(p_i).
\]
Combining this with \eqref{intineb},
\begin{equation}
\begin{aligned}\label{preciseinb}
d_{int}(\Sigma)&\le\frac{2}{\rho}\sum_is(p_i)\\
&<\frac{2}{\rho}\delta^{-1}\sum_i
\left(\int_{B(p_i,s(p_i))\cap\Sigma}|H|dv_\Sigma+L(B(p_i,s(p_i))\cap \partial\Sigma)\right)\\
&\le \frac{2}{\rho}\delta^{-1}
\left(\int_{\Sigma}|H|dv_\Sigma+L(\partial\Sigma)\right),
\end{aligned}
\end{equation}
where $\delta$ is chosen as in Lemma~\ref{lemb}.
Letting $\rho\to\frac{1}{2}-$ proves the theorem.
\end{proof}

\textbf{Note added in theorems}.
In March 2023, Ma-Wu \cite{MW} extended the Brendle's Sobolev inequality
under a weaker curvature condition of ambient space. In fact they showed
that Brendle's Sobolev type inequality could hold when the $k$-th intermediate
Ricci curvature of ambient space is nonnegative. Using this, we can similarly
prove Theorems \ref{Main} and \ref{Main2} under their curvature condition.

\textbf{Acknowledgements}.
The author thanks the anonymous referees for making valuable comments and suggestions
which helped to improve the presentation of this work.

\end{document}